\documentclass{article}
\usepackage{latexsym,amssymb,amsmath,amsfonts,pictex,enumerate,amsthm,dsfont}
\usepackage{ mathrsfs }
\usepackage{tikz}
\usetikzlibrary{decorations.markings}
\usepackage{scalerel,stackengine}
\usepackage[active]{srcltx}

\theoremstyle{definition}
\newtheorem{theorem}{Theorem}

\newtheorem{lemma}[theorem]{Lemma}

\stackMath
\newcommand\reallywidehat[1]{
\savestack{\tmpbox}{\stretchto{
  \scaleto{
    \scalerel*[\widthof{\ensuremath{#1}}]{\kern-.6pt\bigwedge\kern-.6pt}
    {\rule[-\textheight/2]{1ex}{\textheight}}
  }{\textheight}
}{0.5ex}}
\stackon[1pt]{#1}{\tmpbox}
}
\title{Bounded point derivations on Campanato spaces}
\author{Evan Abshire\thanks{Email: abshire17@marshall.edu} and Stephen Deterding\thanks{Email: deterding@marshall.edu}}
\date{Marshall University, 1 John Marshall Dr. Huntington WV, 25755 USA}

\begin{document}

\maketitle

\begin{abstract}
    Let $X$ be a compact subset of the complex plane and $x \in X$. A necessary and sufficient condition is given in terms of Hausdorff contents for the existence of a bounded point derivation at $x$ on the space of vanishing Campanato functions that are analytic in a neighborhood of $X$. This generalizes many known conditions for the existence of bounded point derivations on other function spaces. 
\end{abstract}

\section{Introduction}

This paper concerns questions about the smoothness of functions at boundary points of subsets of the complex plane. Let $X$ be a compact subset of the complex plane and let $R(X)$ denote the uniform closure of rational functions with poles off $X$. $R(X)$ is widely studied in the theory of complex approximation; for example, Runge's theorem states that if $A(X)$ denotes the space of functions that are analytic in a neighborhood of $X$, then $R(X) = A(X)$; that is, every analytic function on $X$ can be uniformly approximated by rational functions with poles off $X$. Every function in $R(X)$ is differentiable at an interior point of $X$, but in general, the functions in $R(X)$ are not differentiable at the boundary points of $X$; however, in many cases the functions in $R(X)$ possess a greater degree of smoothness than what otherwise would be expected. 

\bigskip

One such example is that $R(X)$ may admit a bounded point derivation at a boundary point $x$. For a non-negative integer $t$, we say that $R(X)$ admits a $t$-th order bounded point derivation at $x$ if there exists a constant $C>0$ such that

\begin{align*}
    |f^{(t)}(x)| \leq C ||f||_{\infty}
\end{align*}

\bigskip
\noindent for all rational functions $f$ with poles off $X$. Here $||\cdot||_{\infty}$ denotes the uniform norm on $X$.

\bigskip

Bounded point derivations play an important role in the theory of rational approximation. Suppose that $\{f_j\}$ is a sequence of rational functions with poles off $X$ that converges to a limit function $f$ on $X$. If $x$ is an interior point of $X$ then the sequence of derivatives $\{f_j'(x)\}$ converges uniformly to $f'(x)$; however, if $x$ is a boundary point then the sequence of derivatives might not converge at all. Nevertheless, if $R(X)$ admits a bounded point derivation at $x$, then the sequence of derivatives will converge and one can define a derivative for $f$ at $x$ as $f'(x) = \displaystyle \lim_{j\to \infty} f_j'(x)$.

\bigskip 

Bounded point derivations can be defined for other spaces of functions as well. Let $X$ be a compact subset of $\mathbb{C}$ and let $U$ be an open subset of $\mathbb{C}$. Some spaces on which bounded point derivations have been studied include $A_{\alpha}(U)$, the space of little Lipschitz functions of order $\alpha$ that are analytic on $U$ \cite{Lord}, $A_0(X)$, the space of VMO functions that are analytic on $X$ \cite{Deterding}, and $A^s(U)$, the functions on $U$ in the small negative Lipschitz space that are analytic on $U$ \cite{O'Farrell2018}. The focus of this paper is on bounded point derivations on Campanato spaces. Campanato spaces include spaces of Lipschitz functions, functions of bounded mean oscillation, and functions in the negative Lipschitz space as special cases and thus can be used to generalize these results.

\bigskip

It is conjectured that for each space there are necessary and sufficient conditions for the existence of bounded point derivations given in terms of an appropriate capacity. Since there is no general theory, the conditions must be verified on a case by case basis; however, since the Campanato spaces contain $A_{\alpha}(U)$, $A_0(X)$, and $A^s(U)$ determining the conditions for the existence of bounded point derivations in the space of analytic functions on the Campanato spaces verifies the conditions for these other spaces as well. Proving the following theorem is thus the principal focus of this paper. (See Section 2 for relevant definitions.)

\begin{theorem}
\label{main}
    Suppose $t$ is a non-negative integer, $1 \leq p < \infty$, and let $\lambda \geq 0$ also satisfy $2-p < \lambda < 2+p$. Let $X$ be a compact subset of $\mathbb{C}$ and let $A_{p, \lambda}(X)$ denote the space of functions in the vanishing Campanato space $V\mathscr{L}^{p, \lambda}(\mathbb{C})$ that are also analytic in a neighborhood of $X$. Let $A_n(x)$ denote the annulus $\{z: 2^{-(n+1)} \leq |z-x| \leq 2^{-n}\}$. Then $A_{p, \lambda}(X)$ admits a $t$-th order bounded point derivation at $x$ if and only if

\begin{align*}
    \sum_{n=1}^{\infty} 2^{(t+1)n} M_*^{1+\frac{\lambda-2}{p}}(A_n(x) \setminus X) < \infty,
\end{align*}

\bigskip
\noindent where $M_*^{1+\frac{\lambda-2}{p}}$ denotes lower $(1+\frac{\lambda-2}{p})$-dimensional Hausdorff content.
\end{theorem}

Theorem \ref{main} is similar to other existence theorems for bounded point derivations. Lord and O'Farrell \cite[Theorem 1.2]{Lord} proved the following theorem for the case of Lipschitz approximation. 

\begin{theorem}
\label{Lord}
Suppose $U \subseteq \mathbb{C}$ is bounded and open,  $0< \alpha<1$ and $t$ is a non-negative integer. Let $A_{\alpha}(U)$ denote the space of functions in the little Lipshitz class of order $\alpha$ that are analytic on $U$, $x \in \partial U$, and let $A_n(x)$ denote the annulus $\{z: 2^{-(n+1)}\leq |z-x|\leq2^{-n}\}$. Then $A_{\alpha}(U)$ admits a $t$-th order bounded point derivation at $x$ if and only if 

\begin{equation*}
    \sum_{n=1}^{\infty} 2^{(t+1)n} M_*^{1+\alpha}(A_n(x) \setminus U) < \infty,
\end{equation*}

\bigskip

\noindent where $M_*^{1+\alpha}$ denotes lower $(1+\alpha)$-dimensional Hausdorff content.

\end{theorem}

\bigskip
\noindent The second author proved the following theorem for the case of BMO approximation \cite[Theorem 1]{Deterding}. 

\begin{theorem}
\label{BMO}

Let $X$ be a compact subset of $\mathbb{C}$ with the property that every relatively open subset of $X$ has positive area, let $t$ be a non-negative integer and let $A_0(X)$ denote the space of $VMO(\mathbb{C})$ functions that are analytic on a neighborhood of $X$. Choose $x \in \partial X$ and let $A_n(x)$ denote the annulus  $\{2^{-(n+1)} \leq |z-x| \leq 2^{-n}\}$. Then $A_0(X)$ admits a $t$-th order bounded point derivation at $x$ if and only if 

\begin{equation*}
    \sum_{n=1}^{\infty} 2^{(t+1)n} M^1_*(A_n(x) \setminus X) < \infty,
\end{equation*}

\bigskip

\noindent where $M^1_*$ denotes lower 1-dimensional Hausdorff content.

\end{theorem}

\bigskip
\noindent O'Farrell has also proven a similar theorem involving negative Lipschitz classes \cite[Theorem 3.7]{O'Farrell2018}. 

\begin{theorem}
\label{O'Farrell}
    Suppose $0 < \beta < 1$ and $s= \beta-1$. Let $U \subseteq \mathbb{C}$ be a bounded open set and $x \in \partial U$. Let $A^s_x(U)$ denote the functions on $U$ in the small negative Lipschitz space that are analytic on some neighborhood of $x$ and let $A_n(x)$ denote the annulus $ \{2^{-(n+1)} \leq |z-x| \leq 2^{-n}\}$. Then $A^s_x(U)$ admits a $t$-th order bounded point derivation at $x$ if and only if

\begin{align*}
    \sum_{n=1}^{\infty} 2^{n(t+1)} M^{\beta}_*(A_n(x) \setminus U) < \infty,
\end{align*}

\bigskip

\noindent where $M^{\beta}_*$ denotes lower $\beta$-dimensional Hausdorff content.
    
\end{theorem}

\bigskip
\noindent In the next section, it will be demonstrated that these three theorems are special cases of Theorem \ref{main}. In particular, Theorem \ref{Lord} is the case of $\lambda = 2+ p\alpha$, Theorem \ref{BMO} is the case of $\lambda = 2$, and Theorem \ref{O'Farrell} is the case of $\lambda = 2+ ps$.

\section{Campanato Spaces}

Campanato spaces (also called Morrey-Campanato spaces) were introduced by Campanato in 1963 \cite{Campanato} and generalize spaces of functions of bounded mean oscillation. Let $f \in L^1_{loc}(\mathbb{C})$ and let $B$ be a ball with radius $r$. If $|B|$ denotes the area of $B$, then the mean value of $f$ on $B$, which is denoted by $f_B$ is given by

\begin{align*}
    f_B = \frac{1}{|B|} \int_B |f| dA.
\end{align*}

\bigskip
\noindent Let $1 \leq p < \infty$ and $\lambda \geq 0$. The Campanato seminorm, which we denote by $[f]_{\mathscr{L}^{p,\lambda}}$, generalizes the mean oscillation of $f$ and is given by

\begin{align*}
    [f]_{\mathscr{L}^{p,\lambda}} = \sup_B \left( \frac{1}{r^{\lambda}} \int_B |f(z) - f_B|^p dA  \right)^{\frac{1}{p}}
\end{align*}

\bigskip
\noindent where the supremum is taken over all balls $B$ in $\mathbb{C}$. Equivalently (See \cite[Lemma 5.6.1]{Pick} for proof.) the Campanato seminorm can also be given by 

\begin{align*}
    [f]_{\mathscr{L}^{p,\lambda}} = \sup_B \left( \frac{1}{r^{\lambda}} \inf_{c \in \mathbb{C}} \int_B |f(z) - c|^p dA  \right)^{\frac{1}{p}}.
\end{align*}

\bigskip
\noindent In both definitions of the Campanato seminorm, up to a constant multiple, the supremum can be taken over squares with side length $r$ instead of balls of radius $r$. The Campanato space $\mathscr{L}^{p, \lambda}(\mathbb{C})$ is the space of $L^p$ functions with finite Campanato seminorms. That is,

\begin{align*}
    \mathscr{L}^{p, \lambda}(\mathbb{C}) = \{f \in L^p(\mathbb{C}): [f]_{\mathscr{L}^{p,\lambda}} < \infty\}.
\end{align*}

\bigskip
\noindent $\mathscr{L}^{p, \lambda}(\mathbb{C})$ is a Banach space with norm given by

\begin{align*}
    ||f||_{\mathscr{L}^{p,\lambda}} = [f]_{\mathscr{L}^{p,\lambda}} + ||f||_p,
\end{align*}

\bigskip
\noindent where $||f||_p$ denotes the $L^p$ norm. We also note that if $f \in L^p(\mathbb{C})$, then $f \in \mathscr{L}^{p, \lambda}(\mathbb{C})$ if for each ball $B$ there exists a constant $c(B)$ such that 

\begin{align*}
    \int_B |f(z) - c(B)|^p dA(z) \leq C(f) r^{\lambda},
\end{align*},

\bigskip
\noindent where the constant $C(f)$ depends only on $f$.

\bigskip

An important feature of Campanato spaces is the following coincidence of spaces that comes from the Campanato embedding property \cite[Theorem 5.5.1]{Pick}. If $p, p_1, \lambda, \lambda_1$ are such that $1 \leq p, p_1 < \infty$, $0\leq \lambda, \lambda_1 < \infty$, and $\frac{\lambda_1-2}{p_1} = \frac{\lambda-2}{p}$ then $\mathscr{L}^{p,\lambda}(\mathbb{C}) = \mathscr{L}^{p_1, \lambda_1}(\mathbb{C})$.

\bigskip

The significance of the Campanato spaces is that they contain several well known function spaces as special cases. For $p \in [1, \infty)$, the case of $\lambda = 2$ is BMO$(\mathbb{C})$ the space of functions of bounded mean oscillation, the case $2<\lambda \leq 2+p$ is the space of Lipschitz continuous functions Lip$_{\alpha}(\mathbb{C})$, where $\alpha = \frac{\lambda-2}{p}$, and the case of $\lambda < 2$ corresponds to Morrey spaces \cite[Theorem 5.7.1]{Pick}. Another coincidence with $\lambda<2$ occurs with the negative Lipschitz space Lip$_{\beta}(\mathbb{C})$, where $\beta = \frac{\lambda-2}{p}$\cite{O'Farrell1988}. It should be noted that $\mathscr{L}^{p,\lambda}(\mathbb{C})$ consists only of constant functions when $\lambda>p+2$.

\bigskip

We now define the vanishing Campanato spaces to serve as generalizations of the space of functions of vanishing mean oscillation. Given $f \in L^p(\mathbb{C})$ and $\delta>0$ let 

\begin{align*}
    \Omega_f^{p,\lambda}(\delta) = \sup_B\left\{\left( \frac{1}{r^{\lambda}} \int_B |f(z) - f_B|^p dA  \right)^{\frac{1}{p}}: \textnormal{ radius } B \leq \delta\right\}
\end{align*}

\bigskip
\noindent where the supremum is taken over all balls $B \subseteq \mathbb{C}$ with radius $\leq \delta$. Let $V\mathscr{L}^{p, \lambda}(\mathbb{C})$ denote the subspace of functions in $\mathscr{L}^{p, \lambda}(\mathbb{C})$ with the property that $\Omega_f^{p,\lambda}(\delta) \to 0$ as $\delta \to 0$. $V\mathscr{L}^{p, \lambda}(\mathbb{C})$ are the vanishing Campanato spaces. As with the Campanato spaces, up to a constant multiple, the supremum can be taken over squares instead of balls in the definition of the vanishing Campanato spaces.

\bigskip

Like the Campanato spaces, the vanishing Campanato spaces include some well known function spaces as special cases. When $\lambda = 2$, $V\mathscr{L}^{p, \lambda}(\mathbb{C})$ coincides with $VMO(\mathbb{C})$, the space of functions of vanishing mean oscillation. Likewise when $2<\lambda \leq 2+p$, $V\mathscr{L}^{p, \lambda}(\mathbb{C})$ coincides with the little Lipschitz class lip$_{\alpha}(\mathbb{C})$ where $\alpha = \frac{\lambda-2}{p}$.

\bigskip

If $X$ is a compact set with the property that every relatively open subset of $X$ has positive area, then we define $\mathscr{L}^{p,\lambda}(X) = \{f|_X : f \in \mathscr{L}^{p,\lambda}(\mathbb{C})\}$ and $V\mathscr{L}^{p,\lambda}(X) = \{f|_X : f \in V\mathscr{L}^{p,\lambda}(\mathbb{C})\}$. Let $[f]_{\mathscr{L}^{p,\lambda}(X)} = \displaystyle \inf [F]_{\mathscr{L}^{p,\lambda}}$, where the infimum is taken over all functions $F$ such that $F =f$ on $X$. $[f]_{\mathscr{L}^{p,\lambda}(X)}$ is a seminorm on $\mathscr{L}^{p,\lambda}(X)$, which vanishes only at the constant functions. If we let $\displaystyle ||f||_{\mathscr{L}^{p,\lambda}(X)} = [f]_{\mathscr{L}^{p,\lambda}(X)} + ||f||_{L^p(X)}$, then $||f||_{\mathscr{L}^{p,\lambda}(X)}$ defines a norm on $\mathscr{L}^{p,\lambda}(X)$.

\bigskip

Let $A_{p,\lambda}(X)$ denote the space of $V\mathscr{L}^{p, \lambda}(\mathbb{C})$ functions that are analytic in a neighborhood of $X$. Suppose $x$ is a point on the boundary of $X$. Then $A_{p,\lambda}(X)$ admits a $t$-th order bounded point derivation at $x$ if there is a constant $C$ such that

\begin{align*}
    |f^{(t)}(x)| \leq C ||f||_{\mathscr{L}^{p,\lambda}(X)}
\end{align*}

\bigskip
\noindent for all functions $f \in A_{p,\lambda}(X)$. 

\section{Hausdorff Content}

It is conjectured that for each function space, the conditions for the existence of bounded point derivations are given in terms of a certain capacity. The appropriate capacity for studying bounded point derivations on $A_{p, \lambda}(X)$ is lower $(1+\frac{\lambda-2}{p})$-dimensional Hausdorff content, which is defined as follows. A measure function is an increasing function $h(t)$, $t \geq 0$ such that $h(t) \to 0$ as $t \to 0$. Given a measure function $h$ and a set $E \subseteq \mathbb{C}$, let

\begin{align*}
    M^h(E) = \inf \sum_j h(r_j),
\end{align*}

\bigskip

\noindent where the infimum is taken over all countable coverings of $E$ by squares with side length $r_j$. Let $\alpha>0$. The lower $\alpha$-dimensional Hausdorff content of $E$ is denoted by $M_*^{\alpha}(E)$ and defined by

\begin{align*}
    M_*^{\alpha}(E) = \sup M^h(E)
\end{align*}

\bigskip
\noindent where the supremum is taken over all measure functions $h$ with $h(t) \leq t^{\alpha}$ and $t^{-\alpha}h(t) \to 0$ as $t \to 0^+$. Furthermore up to a constant multiplicative bound, the infimum can be taken over dyadic squares, or balls of radius $r_j$. It follows directly from the definition that lower $\alpha$-dimensional Hausdorff content is monotone; that is, if $E \subseteq F$ then $M_*^{\alpha}(E) \leq M_*^{\alpha}(F)$. Another property of Hausdorff content is that if $B_r$ is a ball of radius $r$ then $M_*^{\alpha}(B_r) = r^{\alpha}$.

\bigskip

We now review Frostman's lemma (See \cite[pg.62]{Garnett} for proof.), which is a key result in relating Hausdorff content and measure.

\begin{lemma}
    Let $h$ be a measure function and let $K \subseteq \mathbb{C}$ be a set with positive lower $\alpha$-dimensional Hausdorff content. Then there is a Borel measure $\nu$ with support on $K$ such that 

    \begin{enumerate}
        \item $\nu(B) \leq C h(r)$ for all balls $B$ with radius $r$.
        \bigskip
        \item $\nu(K) \geq M_*^{\alpha}(K)$.
    \end{enumerate}
\end{lemma}

\section{Preliminary Results}

In this section, we prove some key lemmas that will be used in the proof of Theorem \ref{main}. Our first result is of independent interest, as it extends a result of Kaufman \cite[Theorem (b)]{Kaufman} to Campanato spaces.

\begin{theorem}
\label{analytic}
Suppose $1 \leq p < 2$ and $2-p < \lambda \leq 2+p$. Let $S$ be a compact set of positive $1+\frac{\lambda-2}{p}$-measure. Then there is a function $g$ analytic off $S$ in $\mathscr{L}^{p, \lambda}(\mathbb{C})$ with Taylor expansion $z^{-1}+\ldots$ at infinity.

\end{theorem}

\bigskip
\noindent \textit{Note that the condition that $S$ has positive $1+\frac{\lambda-2}{p}$-measure is satisfied by subsets of $\mathbb{C}$ with positive area.}

\begin{proof}
The proof follows that in \cite{Kaufman}. To simplify computations in the proof we will rewrite $1+\frac{\lambda-2}{p}$ as $\frac{p+\lambda-2}{p}$. By Frostman's lemma there is a measure $\nu$ supported on $S$ such that $\nu(B(z,r)) \leq C r^{\frac{p+\lambda  -2}{p}}$ for every ball $B$ of radius $r>0$. Let 

\begin{align*}
    g(z) = \int (\zeta-z)^{-1} d\nu(\zeta).
\end{align*}

\bigskip
\noindent Then $g$ is analytic off $S$ and $g(z) = z^{-1} + \ldots$. To prove that $g \in \mathscr{L}^{p, \lambda}(\mathbb{C})$ let $B= B(w,r)$ be a ball centered at $w$ with radius $r$ and $B^* = B(w,2r)$. Now define

\begin{align*}
    g_1(z) = \int_{B^*} (\zeta-z)^{-1} d\nu(\zeta)
\end{align*}

\bigskip
\noindent and let $g_2(z) = g(z)-g_1(z)$.

\bigskip

Let $q = \frac{p}{p-1}$. Then it follows from H\"{o}lder's inequality and Fubini's theorem that 

\begin{align*}
    \int_B |g_1(z)|^p dA(z) &\leq \int_B \left( \int_{B^*} |\zeta-z|^{-p} d\nu(\zeta)\right) \left( \int_{B^*} 1 d\nu(\zeta) \right)^{\frac{p}{q}} dA(z)\\
    &\leq C r^{\frac{p+\lambda-2}{q}} \int_{B^*} \int_B |\zeta-z|^{-p} dA(z) d\nu(\zeta)\\
    &\leq C r^{\frac{p+\lambda-2}{q}} r^{\frac{p+\lambda-2}{p}}r^{2-p} = C r^{\lambda}.
\end{align*}

\bigskip

Moreover,

\begin{align*}
    \int_B |g_2(z)-g_2(w)|^p dA(z) &= \int_B \left| \int_{\mathbb{C}\setminus B^*} [(\zeta-z)^{-1}-(\zeta-w)^{-1}] d\nu(\zeta) \right|^p dA(z)\\
    &\leq \int_B \left( \int_{\mathbb{C}\setminus B^*} |\zeta-w|^{-1} d\nu(\zeta)\right)^p dA(z)\\
    &\leq (2r)^p\int_B \left( \int_{\mathbb{C}\setminus B^*} |\zeta-w|^{-2} d\nu(\zeta)\right)^p dA(z)\\
    &\leq (2r)^p \pi r^2 \left((2r)^{-2} (2r)^{\frac{p+\lambda-2}{p}}\right)^p = Cr^{\lambda}
\end{align*}

\bigskip
\noindent Thus 

\begin{align*}
    \int_B |g(z)-g_2(w)|^p dA(z) \leq C r^{\lambda}
\end{align*}

\bigskip
\noindent and hence $g \in \mathscr{L}^{p, \lambda}(\mathbb{C})$.
    
\end{proof}

\bigskip

Next we verify some important properties of a function that is crucial to the proof of the necessity of the criterion in Theorem \ref{main}.

\begin{lemma}
\label{lem1}
Let $X$ be a compact subset of $\mathbb{C}$ and suppose $1 \leq p <2$ and $2-p < \lambda \leq 2+p$. Let $A_n$ denote the annulus $\{2^{-(n+1)} \leq |z| \leq 2^{-n} \}$ and suppose $\nu_n$ is a measure on $A_n \setminus X$ with the following properties:

\begin{enumerate}
    \item $\nu_n(B_r) \leq \epsilon_n r^{1+ \frac{\lambda-2}{p}}$ for all balls $B$ with radius $r$.
    \bigskip
    \item $\int \nu_n = C \epsilon_n M_*^{1+ \frac{\lambda-2}{p}}(A_n \setminus X)$.
\end{enumerate}

\bigskip
\noindent Let $t$ be a non-negative integer and define

\begin{align*}
    f_n(z) = \int \left( \frac{\zeta}{|\zeta|} \right)^{t+1} \frac{d\nu_n(\zeta)}{\zeta-z}.
\end{align*}

\bigskip
\noindent Let $\displaystyle f_{n,B} = \frac{1}{|B|} \int_B |f_n| dA$. Then 

\begin{enumerate}
    \item $[f_n]_{\mathscr{L}^{p,\lambda}} \leq C \epsilon_n$.
    \bigskip
    \item If $B \subseteq A_k$ is a ball of radius $r$ and $k \neq n-1, n,$ or $n+1$, then

    \begin{align*}
       \left( \frac{1}{r^{\lambda}}  \int_B |f_n(z)|^p dA \right)^{\frac{1}{p}} \leq C \epsilon_n 2^{n(1+ \frac{\lambda-2}{p})} M_*^{1+ \frac{\lambda-2}{p}}(A_n \setminus X).
    \end{align*}
    \bigskip
     \item If $B \subseteq A_k$ is a ball of radius $r$ and $k \neq n-1, n,$ or $n+1$, then

    \begin{align*}
      \left(  \frac{1}{r^{\lambda}}  \int_B |f_{n,B}|^p dA \right)^{\frac{1}{p}}\leq C \epsilon_n 2^{n(1+ \frac{\lambda-2}{p})} M_*^{1+ \frac{\lambda-2}{p}} (A_n \setminus X).
    \end{align*}

    \bigskip
    \item If $||\cdot||_p$ denotes the $L^p$ norm on $\mathbb{C}$, then 

    \begin{align*}
        ||f_n||_p \leq C \epsilon_n 2^{n(1+ \frac{\lambda-2}{p})} M_*^{1+ \frac{\lambda-2}{p}} (A_n \setminus X).
    \end{align*}

\end{enumerate}

\end{lemma}

\begin{proof}
   As in the previous proof, to simplify calculations we will write $1+\frac{\lambda-2}{p}$ as $\frac{p+\lambda-2}{p}$. Let $q = \frac{p}{p-1}$. The proof of the first proposition is done is the same way as the proof of Theorem \ref{analytic}. To prove the second proposition we first note that it follows from H\"{o}lder's inequality that

\begin{align}
\label{eq1}
    \int_B |f_n(z)|^p dA(z) &\leq \int_B \left(\int \frac{d\nu_n(\zeta)}{|\zeta-z|} \right)^p dA(z) \\
    &= \int_B \left(\int \frac{d\nu_n(\zeta)}{|\zeta-z|^{\frac{2-\lambda}{p}} |\zeta-z|^{\frac{p+\lambda-2}{p}} } \right)^p dA(z) \nonumber \\
    &\leq \int_B \left(\int \frac{d\nu_n(\zeta)}{|\zeta-z|^{2-\lambda}  } \right) \left(\int \frac{d\nu_n(\zeta)}{|\zeta-z|^{\frac{p+\lambda-2}{p-1}} } \right)^{\frac{p}{q}}  dA(z). \nonumber
\end{align}

\bigskip
\noindent We first evaluate the second inner integral. Because $\nu_n$ has support on $A_n \setminus X$ and $k \neq n-1, n,$ or $n+1$, it follows that $|\zeta-z| \geq 2^{-(n+2)}$ and hence

\begin{align*}
    \left(\int \frac{d\nu_n(\zeta)}{|\zeta-z|^{\frac{p+\lambda-2}{p-1}} } \right)^{\frac{p}{q}} \leq C 2^{n(p+\lambda-2)} \left(\epsilon_n M_*^{\frac{p+\lambda-2}{p}}(A_n \setminus X) \right)^{\frac{p}{q}}.
\end{align*}

\bigskip
\noindent For the remaining two integrals, we can use Fubini's theorem to interchange the order of integration. Hence

\begin{align*}
    \int_B |f_n(z)|^p dA(z) &\leq C 2^{n(p+ \lambda -2)} \left(\epsilon_n M_*^{\frac{p+ \lambda -2}{p}}(A_n \setminus X) \right)^{\frac{p}{q}} \int \int_B |\zeta-z|^{\lambda-2} dA(z) d\nu_n(\zeta)\\
    &\leq C 2^{n(p+ \lambda -2)} \left(\epsilon_n M_*^{\frac{p+ \lambda -2}{p}}(A_n \setminus X) \right)^{\frac{p}{q}} r^{\lambda} \epsilon_n M_*^{\frac{p+ \lambda -2}{p}}(A_n \setminus X).
\end{align*}

\bigskip
\noindent Thus

\begin{align*}
  \left( \frac{1}{r^{\lambda}} \int_B |f_n(z)|^p dA(z)\right)^{\frac{1}{p}} &\leq C 2^{n(\frac{p+\lambda-2}{p})} \left(\epsilon_n M_*^{\frac{p+\lambda-2}{p}}(A_n \setminus X) \right)^{\frac{1}{q}} \left(\epsilon_n M_*^{\frac{p+\lambda-2}{p}}(A_n \setminus X) \right)^{\frac{1}{p}}\\
  &= C \epsilon_n 2^{n(\frac{p+\lambda-2}{p})} M_*^{\frac{p+\lambda-2}{p}}(A_n \setminus X).
\end{align*}

\bigskip

Similarly, we prove the third proposition. We first observe that

\begin{align*}
    \int_B |f_{n,B}|^p dA &= \int_B \left| \frac{1}{|B|} \int_B f_n(z) dA(z) \right|^p dA\\
    &\leq C r^{2-2p} \left(\int_B \int \frac{d\nu_n(\zeta)}{|\zeta-z|} dA(z) \right)^p.
\end{align*}

\bigskip
\noindent Then by applying H\"{o}lder's inequality  we obtain

\begin{align*}
    r^{2-2p} \left(\int_B \int \frac{d\nu_n(\zeta)}{|\zeta-z|} dA(z) \right)^p &\leq r^{2-2p} \left(\int_B \left(\int \frac{d\nu_n(\zeta)}{|\zeta-z|}  \right)^p  dA(z)\right) \left( \int_B 1 dA \right)^{\frac{p}{q}}\\
    &= C \int_B \left(\int \frac{d\nu_n(\zeta)}{|\zeta-z|}  \right)^p  dA(z).
\end{align*}

\bigskip
\noindent Since this is the same integral as in \eqref{eq1}, repeating the same calculations yields 

\begin{align*}
  \left( \frac{1}{r^{\lambda}} \int_B |f_{n,B}|^p dA(z)\right)^{\frac{1}{p}} \leq C \epsilon_n 2^{n(\frac{p+\lambda-2}{p})} M_*^{\frac{p+\lambda-2}{p}}(A_n \setminus X).
\end{align*}

  \bigskip

  To prove the fourth proposition, we use the same techniques as before. By H\"{o}lder's inequality

  \begin{align*}
      \int_X |f_n|^p dA(z) &\leq \int_X \left(\int \frac{d\nu_n(\zeta)}{|\zeta-z|}\right)^p dA(z)\\
      &\leq \int_X \left(\int \frac{d\nu_n(\zeta)}{|\zeta-z|^p}\right) \left(\int 1 d\nu_n\right)^{\frac{p}{q}} dA(z)\\
      &= C \left(\epsilon_n M_*^{\frac{p+\lambda-2}{p}}(A_n \setminus X) \right)^{\frac{p}{q}} \int_X \left(\int \frac{d\nu_n(\zeta)}{|\zeta-z|^p}\right) dA(z).
  \end{align*}

   Then by Fubini's theorem,

  \begin{align*}
      \int_X |f_n|^p dA(z) &\leq C \left(\epsilon_n M_*^{\frac{p+\lambda-2}{p}}(A_n \setminus X) \right)^{\frac{p}{q}} \int \int_X \frac{1}{|\zeta-z|^p} dA(z) d\nu_n(\zeta).
  \end{align*}

  \bigskip
  \noindent Let $B_r(\zeta)$ be a ball of radius $r$ centered at $\zeta$ such that $X \subseteq B_r(\zeta)$. Then since $p<2$,

\begin{align*}
      \int \int_X \frac{1}{|\zeta-z|^p} dA(z) d\nu_n &\leq \int \int_{B_r(\zeta)} \frac{1}{|\zeta-z|^p} dA(z) d\nu_n(\zeta)\\
      &\leq C  \epsilon_n M_*^{\frac{p+\lambda-2}{p}}(A_n \setminus X).
  \end{align*}

  \bigskip
  \noindent Thus 

  \begin{align*}
      \left(\int_X |f_n|^p dA(z)\right)^{\frac{1}{p}} &\leq C \left(\epsilon_n M_*^{\frac{p+\lambda-2}{p}}(A_n \setminus X) \right)^{\frac{1}{q}} \left(\epsilon_n M_*^{\frac{p+\lambda-2}{p}}(A_n \setminus X) \right)^{\frac{1}{p}}\\
  &= C \epsilon_n  M_*^{\frac{p+\lambda-2}{p}}(A_n \setminus X).
  \end{align*}
    
\end{proof}

\section{Proof of Sufficiency}

We first prove the case of sufficiency in Theorem \ref{main}.

\begin{theorem}
Suppose that $X$ is a compact subset of $\mathbb{C}$ and let $A_n(x)$ denote the annulus $\{ 2^{-(n+1)} \leq |z-x| \leq 2^{-n} \}$. Let $t$ be a non-negative integer, $1 \leq p < \infty$, and let $\lambda \geq 0$ also satisfy $2-p < \lambda < 2+p$. If

\begin{align*}
    \sum_{n=0}^{\infty} 2^{(t+1)n} M_*^{1+ \frac{\lambda-2}{p}}(A_n(x) \setminus X) < \infty,
\end{align*}

\bigskip
\noindent then $A_{p, \lambda}(X)$ admits a $t$-th order bounded point derivation at $x$.
    
\end{theorem}

\begin{proof}
    We first choose $f \in A_{p,\lambda}(X)$ so that $[f]_{\mathscr{L}^{p,\lambda}(X)} = 1$. Choose $K_n \subseteq A_n \setminus X$ so that $f$ is analytic on $A_n \setminus K_n$. Fix $n$ and let $\{Q_j\}$ be a covering of $K_n$ by dyadic squares with no overlap except at the boundaries and let $r_j$ denote the side length of $Q_j$. Let $Q_j^* = \frac{3}{2} Q_j$, the square with side length $\frac{3}{2}r_j$ and the same center as $Q_j$, and let $D_n = \bigcup Q_j^*$. Then it follows from the Cauchy integral formula that

\begin{align*}
    f^{(t)}(x) = \frac{t!}{2 \pi i} \sum_n \int_{\partial D_n} \frac{f(z)}{(z-x)^{t+1}} dz.
\end{align*}

\bigskip
\noindent For each square $Q_j$ construct a smooth function $\phi_j$ with support on $Q_j^*$ such that $||\nabla \phi_j||_{\infty} \leq C r_j^{-1}$, and $\sum_j \phi_j = 1$ on a neighborhood of $\bigcup Q_j$. Such a construction is found in \cite[Theorem 3.1]{Harvey}. Let $\phi = 1- \sum_j \phi_j$. Then $\phi = 1$ on $\partial D_n$ and by Green's Theorem,

\begin{align*}
    \left|\frac{t!}{2 \pi i}  \int_{\partial D_n} \frac{f(z)}{(z-x)^{t+1}} dz\right| &= \left|\frac{t!}{2 \pi i} \int_{\partial D_n} \frac{f(z) \phi(z)}{(z-x)^{t+1}} dz\right|\\
    &= \left|\frac{t!}{\pi}  \int_{ D_n} \frac{f(z) }{(z-x)^{t+1}} \frac{\partial \phi}{\partial \overline{z}}dA\right|\\
     &\leq \frac{t!}{\pi} \sum_j \left| \int_{Q_j^*} \frac{f(z) }{(z-x)^{t+1}} \frac{\partial \phi_j}{\partial \overline{z}}dA \right|.
\end{align*}

\bigskip
\noindent

Moreover, $\int_{Q_j^*} (z-x)^{-(t+1)} \frac{\partial \phi_j}{\partial \overline{z}} dA = \int_{\partial Q_j^*} (z-x)^{-(t+1)} \phi_j(z) dz = 0$, and hence

\begin{align*}
    \sum_j \left| \int_{Q_j^*} \frac{f(z) }{(z-x)^{t+1}} \frac{\partial \phi_j}{\partial \overline{z}}dA \right| &\leq  \sum_j  \int_{Q_j^*} \frac{|f(z) - f_{Q_j^*}| }{|z-x|^{t+1}} \left|\frac{\partial \phi_j}{\partial \overline{z}}\right|dA. 
\end{align*}

\bigskip
\noindent Thus by H\"{o}lder's inequality,

\begin{align*}
    \sum_j  \int_{Q_j^*} \frac{|f(z) - f_{Q_j^*}| }{|z-x|^{t+1}} \left|\frac{\partial \phi_j}{\partial \overline{z}}\right|dA &\leq \sum_j  2^{n(t+1)} \left(\int_{Q_j^*} |f(z)-f_{Q_j^*}|^p dA\right)^{\frac{1}{p}} \left( \int_{Q_j^*} \left|\frac{\partial \phi_j}{\partial \overline{z}}\right|^q dA \right)^{\frac{1}{q}}\\
    &\leq C 2^{n(t+1)} \sum_j r_j^{1+ \frac{\lambda-2}{p}} \left( \frac{1}{r^{\lambda}}\int_{Q_j^*} |f(z)-f_{Q_j^*}|^p dA\right)^{\frac{1}{p}}.
\end{align*}

\bigskip
\noindent Since $f \in V\mathscr{L}^{p, \lambda}(\mathbb{C})$ it follows that $h(t) = t^{1+ \frac{\lambda-2}{p}} \Omega_f^{p, \lambda}(\frac{3}{2}t)$ is admissible for $M_*^{1+ \frac{\lambda-2}{p}}(K_n)$. Hence by taking the infimum over all such covers $\{Q_j\}$ we have that

\begin{align*}
    \left|\frac{t!}{2 \pi i}  \int_{\partial D_n} \frac{f(z)}{(z-x)^{t+1}} dz\right| \leq C 2^{n(t+1)} M_*^{1+ \frac{\lambda-2}{p}}(K_n).
\end{align*}

\bigskip
\noindent Since $M_*^{1+ \frac{\lambda-2}{p}}$ is monotone, it follows that

\begin{align*}
    |f^{(t)}(x)| &\leq C \sum_{n=1}^{\infty} 2^{n(t+1)}M_*^{1+ \frac{\lambda-2}{p}}(K_n)\\
    &\leq C \sum_{n=1}^{\infty} 2^{n(t+1)}M_*^{1+ \frac{\lambda-2}{p}}(A_n \setminus X)\\
    &\leq C.
\end{align*}

\bigskip
\noindent Now suppose $g \in A_{p,\lambda}(X)$ is analytic on $X$ and let $f = \frac{g}{[g]_{\mathscr{L}^{p,\lambda}(X)}}$. Then $[f]_{\mathscr{L}^{p,\lambda}(X)} = 1$ and hence $|f^{(t)}(x)| \leq C$. Thus $|g^{(t)}(x)| \leq C [g]_{\mathscr{L}^{p, \lambda}(X)}$ and $A_{p, \lambda}(X)$ admits a $t$-th order bounded point derivation at $x$.

\end{proof}

\section{Proof of Necessity}

Finally, we prove the case of necessity in Theorem \ref{main}

\begin{theorem}
\label{necessity}
    Suppose $X$ is a compact subset of $\mathbb{C}$ and let $A_n(x)$ denote the annulus $\{2^{-(n+1)} \leq |z-x| \leq 2^{-n}\}$. Let $t$ be a non-negative integer, $1 \leq p < \infty$, and let $\lambda \geq 0$ also satisfy $2-p < \lambda < 2+p$. If $A_{p, \lambda}(X)$ admits a $t$-th order bounded point derivation at $x$, then

\begin{align*}
    \sum_{n=1}^{\infty} 2^{(t+1)n} M_*^{1+ \frac{\lambda-2}{p}}(A_n(x) \setminus X) < \infty.
\end{align*}
\end{theorem}

\begin{proof}

We will verify the theorem by proving the contrapositive. Furthermore, we can assume that $x=0$ and that $X$ is contained entirely within the closed unit disk. First suppose that $p<2$. If 

\begin{align*}
    \sum_{n=1}^{\infty} 2^{n(t+1)}M_*^{1+ \frac{\lambda-2}{p}}(A_n \setminus X) = \infty,
\end{align*}

\bigskip
\noindent then we can find a decreasing sequence $\epsilon_n \to 0$ such that

\begin{align*}
    \sum_{n=1}^{\infty} 2^{n(t+1)} \epsilon_n M_*^{1+ \frac{\lambda-2}{p}}(A_n \setminus X) = \infty,
\end{align*}

\bigskip
\noindent and $2^{n(t+1)} \epsilon_n M_*^{1+ \frac{\lambda-2}{p}}(A_n \setminus X) \leq 1$ for all $n$.

\bigskip
\noindent By Frostman's lemma, for each $n \in \mathbf{N}$ there exists a positive measure $\nu_n$ supported on $A_n \setminus X$ such that

\begin{enumerate}
    \item $\nu_n(B) \leq \epsilon_n r^{1+ \frac{\lambda-2}{p}}$ for all balls $B$ with radius $r$.
    \bigskip
    \item $\int \nu_n = C \epsilon_n M_*^{1+ \frac{\lambda-2}{p}}(A_n \setminus X)$.
\end{enumerate}

\bigskip
\noindent Let 

\begin{align*}
    f_n(z) = \int \left( \frac{\zeta}{|\zeta|}\right)^{t+1} \frac{dv_n(\zeta)}{\zeta-z}.
\end{align*}

\bigskip
\noindent It follows from Lemma \ref{lem1} that $f_n$ is analytic off $A_n$, $[f_n]_{\mathscr{L}^{p, \lambda}} \leq C \epsilon_n$ and $f_n \in A_{p, \lambda}(X)$. Moreover,

\begin{align*}
    f_n^{(t)}(0) = t! \int \frac{d\nu_n(\zeta)}{|\zeta|^{t+1}},
\end{align*}

\bigskip
\noindent and hence $f_n^{(t)}(0) \geq C 2^{n(t+1)} M_*^{1+ \frac{\lambda-2}{p}}(A_n \setminus X)$. For each $m \in \mathbb{N}$, choose $M$ so that

\begin{align*}
    1 \leq \sum_{n=m}^M 2^{n(t+1)} M_*^{1+ \frac{\lambda-2}{p}}(A_n \setminus X) \leq 2
\end{align*}

\bigskip
\noindent and let 

\begin{align*}
    g_m(z) = \sum_{n=m}^M f_n(z).
\end{align*}

\bigskip
\noindent It follows that there is a nonzero constant for which $g_m^{(t)}(0)$ is bounded below for all $m$. We wish to show that $||g_m||_{\mathscr{L}^{p,\lambda}(X)} \to 0$ as $m \to \infty$.

\bigskip

Let $B$ be a ball of radius $r$ contained in the annulus $A_k = \{z: 2^{-(k+1)} \leq |z| \leq 2^k\}$ and choose $f_n$ with $m \leq n \leq M$. Then it follows from Proposition 1 of Lemma \ref{lem1} that $[f_n]_{\mathscr{L}^{p, \lambda}} \leq C \epsilon_n$. However, if $k \neq n-1, n,$ or $n+1$, then 

\begin{align*}
       \left( \frac{1}{r^{\lambda}}  \int_B |f_n(z)|^p dA \right)^{\frac{1}{p}} \leq C \epsilon_n 2^{n(1+ \frac{\lambda-2}{p})} M_*^{1+ \frac{\lambda-2}{p}}(A_n \setminus X)
    \end{align*}

    \bigskip
    \noindent and 
    
    \begin{align*}
      \left(  \frac{1}{r^{\lambda}}  \int_B |f_{n,B}|^p dA \right)^{\frac{1}{p}}\leq C \epsilon_n 2^{n(1+ \frac{\lambda-2}{p})} M_*^{1+ \frac{\lambda-2}{p}} (A_n \setminus X).
    \end{align*}

\bigskip
\noindent Hence 

\begin{align*}
    [g_m]_{\mathscr{L}^{p,\lambda}} &\leq \sum_{n=m}^M [f_n]_{\mathscr{L}^{p,\lambda}} \\
    &\leq 3C \epsilon_m + \sum_{n=m}^M  \left( \frac{1}{r^{\lambda}}  \int_B |f_n(z)|^p dA \right)^{\frac{1}{p}} + \left(  \frac{1}{r^{\lambda}}  \int_B |f_{n,B}|^p dA \right)^{\frac{1}{p}}\\
    &\leq 3C \epsilon_m + C \sum_{n=m}^M \epsilon_n 2^{n(1+ \frac{\lambda-2}{p})} M_*^{1+ \frac{\lambda-2}{p}} (A_n \setminus X).
\end{align*}

\bigskip
\noindent In addition, by Proposition 4 of Lemma \ref{lem1} we have that

\begin{align*}
    ||g_m||_{p} &\leq \sum_{n=m}^{M} ||f_n||_p\\
    &\leq C \sum_{n=m}^M \epsilon_n 2^{n(1+ \frac{\lambda-2}{p})} M_*^{1+ \frac{\lambda-2}{p}} (A_n \setminus X).
\end{align*}

\bigskip
\noindent Therefore it follows that

\begin{align*}
    ||g_m||_{\mathscr{L}^{p, \lambda}} &\leq 3C \epsilon_m + C \sum_{n=m}^M \epsilon_n 2^{n(1+ \frac{\lambda-2}{p})} M_*^{1+ \frac{\lambda-2}{p}} (A_n \setminus X)\\
&\leq 3C \epsilon_m + C 2^{-m(t-\frac{\lambda-2}{p})} \sum_{n=m}^M \epsilon_n 2^{n(t+1)} M_*^{1+ \frac{\lambda-2}{p}} (A_n \setminus X)\\
&\leq 3C \epsilon_m + C 2^{-m(t-\frac{\lambda-2}{p})}.\\    
\end{align*}

\bigskip
\noindent Hence $||g_m||_{\mathscr{L}^{p, \lambda}} \to 0$ as $m \to \infty$ but $g_m^{(t)}(0)$ is bounded below by a nonzero constant for all $m$. Hence $A_{p,\lambda}$ does not admit a $t$-th order bounded point derivation at $0$, which proves the theorem when $p<2$.

We prove the case when $p \geq 2$ by using the fact that the Campanato space $\mathscr{L}^{p,\lambda}(X)$ coincides with the Campanato space $\mathscr{L}^{p_1,\lambda_1}(X)$ when $\frac{\lambda_1-2}{p_1} = \frac{\lambda-2}{p}$. Given $p\geq2$ and $\lambda>2$, choose $\lambda_1$ so that $2 + \frac{\lambda-2}{p} \leq \lambda_1 < 2+ \frac{2(\lambda-2)}{p}$ and let $p_1 = p \left(\frac{\lambda_1-2}{\lambda-2} \right)$. Then $1 \leq p_1 < 2$ and $\frac{\lambda-2}{p} = \frac{\lambda_1-2}{p_1}$. Thus the result of Theorem \ref{necessity} for the pair $p$ and $\lambda$ follows from the result with the pair $p_1$ and $\lambda_1$.

\bigskip

If $p\geq 2$ and $\lambda<2$, choose $\lambda_1$ so that $2+\frac{2(\lambda-2)}{p}< \lambda_1\leq 2 + \frac{\lambda-2}{p}$ and let $p_1 = p\left( \frac{2-\lambda_1}{2-\lambda}\right)$. Then $1 \leq p_1 < 2$ and $\frac{\lambda-2}{p} = \frac{\lambda_1-2}{p_1}$. Thus the result of Theorem \ref{necessity} for the pair $p$ and $\lambda$ follows from the result with the pair $p_1$ and $\lambda_1$.

\bigskip

When $\lambda = 2$, the Campanato space $\mathscr{L}^{p,2}(X)$ coincides with BMO$(X)$ for all $p$, so the result of Theorem \ref{necessity} for $p \geq 2$ follows from the result with $p<2$ in this case.

\end{proof}

\section{An Example}

Let $A_n = \{ z \in \mathbb{C} : \frac{1}{2^{n+1}} \leq \vert z \vert \leq \frac{1}{2^n}  \}$. For a given value of $n \in \mathbb{N}$, $A_n$ represents an annulus of the complex plane centered at the origin of the complex plane. From each annulus, an open disk is removed, with the constraint that the deleted disks may not sit on the edge of two annuli. More precisely, let $D_n$ denote the open disk deleted from $A_n$. A roadrunner set $X$ is defined as $X = \displaystyle \bigcup_{n=1}^{\infty} [ A_n \setminus D_n ]$. See Figure \ref{roadrunner}.

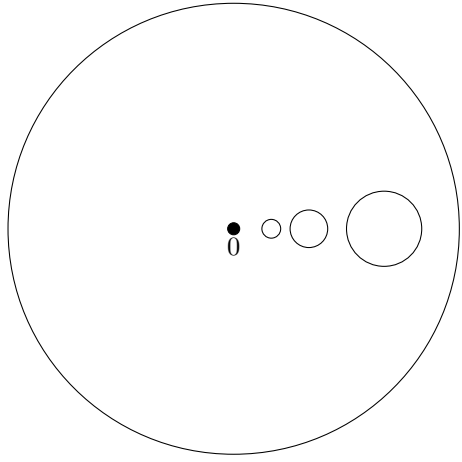
\begin{figure}[t]
\begin{center}
\begin{tikzpicture}[scale=1]

\tikzset{->-/.style={decoration={
  markings,
  mark=at position .5 with {\arrow[scale=2]{>}}},postaction={decorate}}}

   \draw[black] (0,0) circle (3cm);
   \draw [black](2,0) circle (.5cm);
   \draw [black](1,0) circle (.25cm);
   \draw [black](.5,0) circle (.125cm);

   \filldraw[fill=black, draw=black] (0,0) circle (.08 cm);
   \node [below] (x) at (0, 0){$0$};

\end{tikzpicture}
\caption{A roadrunner set}
\label{roadrunner}

\end{center}
\end{figure}

\begin{theorem}
There exists a roadrunner set $X$ such that $A_{p,\lambda}(X)$ admits a bounded point derivation at $0$ for every value of $\lambda$ and $p$, where $1 \leq p < \infty$, $ p-2 < \lambda < p + 2$, and $\lambda \geq 0$.
\end{theorem}

\begin{proof}

    Let $r_n = \frac{1}{n!}$, where $r_n$ is the radius of each deleted disk $D_n$, and let $\epsilon > 0$ be arbitrary. It is known that $A_{p,\lambda}(X)$ admits a bounded point derivation at $0$ if $\displaystyle \sum_{n=1}^{\infty} 4^n  M_*^{1+\frac{\lambda-2}{p}}(A_n(x) \setminus X) < \infty$. From the ratio test, it follows that if
\begin{align*}
  \lim_{n \to \infty} \left\vert \frac{ 4^{n+1} M_*^{ \epsilon} (A_{n+1}(x) \setminus X) }{ 4^n M_*^{ \epsilon} (A_n(x) \setminus X)} \right\vert < 1 ,
\end{align*}
 then $A_{p,\lambda}(X)$ admits a bounded point derivation at $0$ for every value of $\lambda$ and $p$. Since $\displaystyle \lim_{n \to \infty} \frac{1}{n + 1} = 0$, then it follows that $\displaystyle \lim_{n \to \infty} \frac{1}{(n + 1)^{\epsilon}} = 0$. For the given roadrunner set, it follows from the ratio test that 
\begin{align*}
    \lim_{n \to \infty} \left\vert \frac{4^{n+1} (M_*^{\epsilon} (A_{n+1}(x) \setminus X)}{4^n (M_*^{\epsilon} (A_n(x) \setminus X)} \right\vert = \lim_{n \to \infty} \left\vert \frac{ 4 r_{n + 1}^{\epsilon}}{{r_n}^{\epsilon}} \right\vert = \lim_{n \to \infty} \left\vert \frac{4 (n!)}{ (n+1)!} \right\vert = \lim_{n \to \infty} \left\vert \frac{4}{n+1} \right\vert = 0.
\end{align*}
Since the ratio test yields 0 for any arbitrary positive number $\epsilon$, regarding the given roadrunner set, this implies that $\displaystyle \sum_{n=1}^{\infty} 4^n  M_*^{1+\frac{\lambda-2}{p}}(A_n(x) \setminus X) < \infty$ for every value of $p$ and $\lambda$.

\end{proof}

\end{document}